\documentclass{article}

\AtEndDocument{\bigskip{\footnotesize%

T.~M.~Schlank: \textsc{ The Hebrew University of Jerusalem,} \par 
\textsc{\;\;\;\;\;\;\;\;\;\;\;\;\;\;\;\;\;\;\; Department of Mathematics, Jerusalem, Israel.}
\par
\addvspace{\medskipamount}
\textit{\;\;\;\;\;\;\;\;\;\;\;\;\;\;\;\;\;\;\; E-mail address:}  
\texttt{tomer.schlank@gmail.com} \par
\addvspace{\medskipamount}

M.~Prasma: \;\;\;\;\;\textsc{ Faculty of Mathematics, University of Regensburg,} \par 
\textsc{\;\;\;\;\;\;\;\;\;\;\;\;\;\;\;\;\;\;\;  Regensburg, 93040, Germany.} \par\addvspace{\medskipamount}
  \textit{\;\;\;\;\;\;\;\;\;\;\;\;\;\;\;\;\;\;\;\; E-mail address:}  \texttt{mtnprsm@gmail.com}
}}

\usepackage[centertags]{amsmath}
\usepackage[nottoc,numbib]{tocbibind}
\usepackage{hyperref}
\usepackage{amsfonts}
\usepackage{MnSymbol}
\usepackage{amsthm}
\usepackage{newlfont}
\usepackage{amscd}
\usepackage{amsmath}
\usepackage{enumerate}
\usepackage{verbatim}
\usepackage{eucal}
\usepackage{color}
\usepackage[all,2cell]{xy}
\UseAllTwocells
\input xy
\xyoption{2cell}
\xyoption{all}

\newcommand{\pbc}[1][dr]{\save*!/#1-1.8pc/#1:(-1,1)@^{|-}\restore}

\newcommand{\NN}{\mathbb{N}}
\newcommand{\GG}{\mathfrak{G}}

\def\K{\mathcal{K}}
\def\B{\mathcal{B}}
\def\C{\mathcal{C}}
\def\D{\mathcal{D}}
\def\GG{\mathbb{G}}
\def\HH{\mathbb{H}}

\def\PP{\mathbb{P}}
\def\QQ{\mathbb{Q}}
\def\F{\mathcal{F}}
\def\SS{\mathcal{S}}

\def\W{\mathcal{W}}

\def\SS{\mathcal{S}}

\def\P{\mathcal{P}}

\def\K{\mathcal{K}}

\def\W{\mathcal{W}}

\def\Om{{\Omega}}

\def\Del{{\Delta}}

\newcommand{\HSwarrow}{\kern0.05ex\vcenter{\hbox{\Huge\ensuremath{\Swarrow}}}\kern0.05ex}
\newcommand{\hSwarrow}{\kern0.05ex\vcenter{\hbox{\huge\ensuremath{\Swarrow}}}\kern0.05ex}
\newcommand{\LLSwarrow}{\kern0.05ex\vcenter{\hbox{\LARGE\ensuremath{\Swarrow}}}\kern0.05ex}
\newcommand{\LSwarrow}{\kern0.05ex\vcenter{\hbox{\ensuremath{\simeq}}}\kern0.05ex}

\newcommand{\HSearrow}{\kern0.05ex\vcenter{\hbox{\Huge\ensuremath{\Searrow}}}\kern0.05ex}
\newcommand{\hSearrow}{\kern0.05ex\vcenter{\hbox{\huge\ensuremath{\Searrow}}}\kern0.05ex}
\newcommand{\LLSearrow}{\kern0.05ex\vcenter{\hbox{\LARGE\ensuremath{\Searrow}}}\kern0.05ex}
\newcommand{\LSearrow}{\kern0.05ex\vcenter{\hbox{\Large\ensuremath{\Searrow}}}\kern0.05ex}

\newcommand{\HDownarrow}{\kern0.05ex\vcenter{\hbox{\Huge\ensuremath{\Downarrow}}}\kern0.05ex}
\newcommand{\hDownarrow}{\kern0.05ex\vcenter{\hbox{\huge\ensuremath{\Downarrow}}}\kern0.05ex}
\newcommand{\LLDownarrow}{\kern0.05ex\vcenter{\hbox{\LARGE\ensuremath{\Downarrow}}}\kern0.05ex}
\newcommand{\LDownarrow}{\kern0.05ex\vcenter{\hbox{\Large\ensuremath{\Downarrow}}}\kern0.05ex}

\newcommand{\HUparrow}{\kern0.05ex\vcenter{\hbox{\Huge\ensuremath{\Uparrow}}}\kern0.05ex}
\newcommand{\hUparrow}{\kern0.05ex\vcenter{\hbox{\huge\ensuremath{\Uparrow}}}\kern0.05ex}
\newcommand{\LLUparrow}{\kern0.05ex\vcenter{\hbox{\LARGE\ensuremath{\Uparrow}}}\kern0.05ex}
\newcommand{\LUparrow}{\kern0.05ex\vcenter{\hbox{\Large\ensuremath{\Uparrow}}}\kern0.05ex}

\newtheorem{thm}{Theorem}[section]
\newtheorem{cor}[thm]{Corollary}
\newtheorem{lem}[thm]{Lemma}

\newtheorem{pro}[thm]{Proposition}

\theoremstyle{definition}
\newtheorem{define}[thm]{Definition}

\newtheorem{example}[thm]{Example}
\newtheorem{defn}[thm]{Definition}

\newtheorem{notn}[thm]{Notation}

\newtheorem{obs}[thm]{Observation}

\theoremstyle{remark}
\newtheorem{rem}[thm]{Remark}

\DeclareMathOperator{\fib}{fib}

\DeclareMathOperator{\op}{op}
\DeclareMathOperator{\Map}{Map}
\DeclareMathOperator{\red}{red}

\DeclareMathOperator{\Set}{Set}
\DeclareMathOperator{\RelCat}{RelCat}

\DeclareMathOperator{\Cat}{Cat}

\DeclareMathOperator{\Fun}{Fun}

\DeclareMathOperator{\Gpinf}{Gp_h}
\DeclareMathOperator{\Ex}{Ex}

\DeclareMathOperator{\Gp}{Gp}

\DeclareMathOperator{\aut}{aut}

\DeclareMathOperator{\cosk}{cosk}

\DeclareMathOperator{\id}{id}

\DeclareMathOperator{\Fin}{\Gp_h^{\fin}}

\DeclareMathOperator{\GpMod}{GpMod}

\DeclareMathOperator{\loops}{\Omega_{\SS}}

\DeclareMathOperator{\fin}{fin}
\DeclareMathOperator{\fact}{Fact}
\DeclareMathOperator{\hpbc}{\pbc\ar@{}[d]|(0.30){\;\;\;\;\;\;\;\;\;\;\;\sim}}
\DeclareMathOperator{\G-Mod}{G{\hbox{-}}Mod}
\DeclareMathOperator{\con}{conj}

\newcommand{\tgpd}{\kern0.05ex\vcenter{\hbox{\footnotesize\ensuremath{2}}}\kern0.05ex\mathcal{G}pd} 

\def\lrar{\longrightarrow}
\def\llar{\longleftarrow}
\def\hrar{\hookrightarrow}

\def\bar{\overline}

\def\bf{\textbf}

\newcommand\ackname{\textbf{Acknowledgements}}
\if@titlepage
  \newenvironment{acknowledgements}{
      \titlepage
      \null\vfil
      \@beginparpenalty\@lowpenalty
      \begin{center}
        \bfseries \ackname\
        \@endparpenalty\@M
      \end{center}}
     {\par\vfil\null\endtitlepage}
\else
  
\fi

\def\kan{\mathtt{K} an}

\def\kan{\mathcal{S}}

\def\bar{\overline}

\def\bf{\textbf}

\renewcommand{\phi}{\varphi}

\title{Sylow theorems for $\infty$-groups}

\author{Matan Prasma\;\;\; Tomer M. Schlank}

\date{\today}
\begin{document}
\maketitle

\tableofcontents

\begin{abstract}

Viewing Kan complexes as $\infty$-groupoids implies that pointed and connected Kan complexes are to be viewed as $\infty$-groups. A fundamental question is then: to what extent can one ``do group theory" with these objects? In this paper we develop a notion of a \textbf{finite $\infty$-group}: an $\infty$-group whose homotopy groups are all finite. We prove a homotopical analogue of Sylow theorems for finite $\infty$-groups. This theorem has two corollaries: the first is a homotopical analogue of Burnside's fixed point lemma for $p$-groups and the second is a ``group-theoretic" characterisation of finite nilpotent spaces.  

\end{abstract}

\section{Introduction}

One of the earliest insights of homotopy theory was that through the construction of a classifying
space, one can view group theory as the connected components of the theory of pointed, connected, $1$-truncated spaces (i.e. with no homotopy groups above dimension $1$). In his work \cite{Wh}, Whitehead extended this insight by describing pointed, connected $2$-truncated spaces by means of $2$-groups (in his term: crossed modules). This result was further extended by Loday \cite{Lod} to an ``algebraic" characterisation of pointed, connected, $n$-truncated spaces for $n\in\mathbb{N}$.
 
Relying on the classifying space construction, early work of Stasheff, and of Boardman and Vogt \cite{St,BV} demonstrated that weakening the unit, inverse and associativity axioms of a topological group by requiring these to hold only up to coherent homotopy, gives a characterisation of spaces equivalent to a loop space. 

A more categorical perspective by Joyal \cite{Joy} further clarified the picture: viewing spaces (more precisely: Kan complexes) as $\infty$-groupoids implies that pointed connected spaces are to be viewed as $\infty$-groups. 

Indeed, when one loops down a pointed connected $\infty$-groupoid, composition incarnates to concatenation of loops and higher associativity may be pictured as straightening paths, paths between paths and so on. 
An appropriate relative category (and thus an $\infty$-category) of ``spaces equivalent to a loop space" may be defined and its associated $\infty$-category is equivalent to the $\infty$-category of pointed connected spaces. 

To what extent can we ``do group theory" with $\infty$-groups? 
One place to start is finite group theory, and this requires a suitable notion of a ``finite" $\infty$-group. A natural (and prevalent) choice is to take this to mean a pointed connected space homotopy groups are all finite. Such a space is called a \textbf{\boldmath$p$-$\infty$-group} if all its homotopy groups are $p$-groups.  
The purpose of this paper is to show that finite pointed connected spaces admit Sylow theorems, analogous to the classical ones, and to demonstrate how such theorems can be useful for homotopy-theoretic purposes. More precisely:

\begin{define}
Let $p$ be a prime. A map $f:\PP\lrar \GG$ of finite $\infty$-groups is called a \textbf{\boldmath$p$-Sylow map} if for all $n\geq 1$, the map
$$\pi_n(f): \pi_n(\PP) \lrar \pi_n(\GG)$$
is an inclusion of a $p$-Sylow subgroup.
\end{define}

Our analogue of the Sylow theorems will take the following form:

\begin{thm}\label{t:syl}
Let $\GG$ be a finite $\infty$-group and let $p$ be a prime.
\begin{enumerate}
\item The $\infty$-groupoid $\mathtt{P}_p$ spanned by $p$-Sylow maps $\PP\lrar \GG$ is equivalent to the \textbf{set} of $p$-Sylow subgroups of $\pi_1(\GG)$. In particular, $\mathtt{P}_p\neq \emptyset$ and $|\pi_0\mathtt{P}_p|\equiv 1\;(mod\; p)$.
\item
For every finite $p$-$\infty$-group $\HH$, any map $\HH\lrar \GG$ factors as $$\xymatrix{\HH\ar[rr]\ar[dr] && \GG\\ &\PP\ar[ur]& }$$ where $\PP\lrar \GG$ is a $p$-Sylow map.
\item
Every two $p$-Sylow maps are conjugate . 
\end{enumerate}
 
\end{thm}
The notions appearing in the theorem above will be explained in $\S$~\ref{s:Sylow}.

\subsection{Organisation}\label{ss: organization}
After some preliminary discussion on the category of $\infty$-groups, we revisit the theory of $k$-invariants in section~\ref{s:k-invariants}. The main result, Theorem~\ref{t:k-invariants} asserts that a suitably defined $\infty$-category of ``$n^{\text{th}}$-level $k$-invariants" is equivalent to the $\infty$-category of $(n+1)$-truncated spaces. In section~\ref{s:Sylow} we present our analogue of the Sylow theorems for $\infty$-groups which appears in the form of Theorem~\ref{t:syl}. Lastly, in section~\ref{s:applications} we present two elementary applications of the Sylow theorems. The first is an analogue of the Burnside's fixed-point lemma, which states that if a $p$-$\infty$-group coherently acts on a prime-to-$p$ space, then the action admits a homotopy fixed point. The second application gives a characterisation of (pointed, connected) finite nilpotent spaces that is analogous to the characterisation of finite nilpotent groups. 

\begin{ackname}
This work takes a path envisioned by Emmanuel Farjoun and as such carries much of his spirit. We are very grateful for his continuous encouragement and suggestions. The first author was supported by the NWO grant 62001604. 
\end{ackname}

\section{Preliminaries}\label{s:Pre}

Throughout, a \textbf{space} will always mean a Kan complex and we denote by $\SS$ the category of spaces. By default, all constructions in simplicial sets will be adjusted to result in $\SS$ via the Kan replacement functor $\Ex^\infty$. We will indicate explicitly where such adjustments are made.

An \textbf{\boldmath$\infty$-category} will mean a simplicial set satisfying the weak Kan condition as in \cite{Joy} and \cite{Lur09}.
We denote the (large) $\infty$-category of $\infty$-categories by $\Cat_\infty$. If $\C\in\Cat_\infty$ and $X$ is a simplicial set, we let $\Fun(X,\C)$ be the mapping simplicial set; this is an $\infty$-category. 
Let $\C$ be an $\infty$-category. We denote by $h\C$ its \textbf{homotopy category}. A map $f:x\lrar y$ in $\C_1$ is an \textbf{equivalence} if its image in $h\C$ is invertible. If $\D_0\subseteq \C_0$ is a subset of objects of $\C$, then the \textbf{full subcategory} of $\C$ spanned by $\D_0$ is the pullback of simplicial sets

$$\xymatrix{\D\pbc\ar@{^(->}[r] \ar[d] & \C\ar[d]\\ \bar{\D_0}\ar@{^(->}[r] & h\C,}$$

where $\bar{\D_0}$ is the full subcategory of $h\C$ spanned by $\D_0$. In particular, if we let $\iota \C\subseteq \C$ be the maximal Kan complex of $\C$, the $\infty$-groupoid spanned by $\D_0$ is the full subcategory of $\iota \C$ spanned by $\D_0$.

A \textbf{relative category} is a pair $(\C,\W_{\C})$ of a category $\C$ together with a subcategory $\W_{\C}\subseteq \C$ called weak equivalences which contains all isomorphisms. If $(\C,\W_{\C})$ and $(\C',\W_{\C'})$ are relative categories, a \textbf{relative functor} (or map) from $\C$ to $\C'$ is a functor $\C\lrar \C'$ which takes $\W_\C$ into $\W_{\C'}$. 

Let $\Set_\Delta^+$ be the category of (possibly large) \textbf{marked simplicial sets}. The objects of $\Set_\Delta^+$ are pairs $(X,\mathcal{E}_X)$ where $X$ is a simplicial set and $\mathcal{E}_X$ is a subset of $X_1$ (referred to as the \textbf{marked edges}) containing the degenerate edges. The morphisms of $\Set_\Delta^+$ are maps of simplicial sets that preserve the marked edges. The category $\Set_\Delta^+$ admits a model structure, called the \textbf{Cartesian model structure} (see \cite[\S 3.1]{Lur09}), in which: 
\begin{enumerate}
\item the fibrant objects are $\infty$-categories with equivalences being the marked edges;
\item cofibrations are monomorphisms;
\item weak equivalences between fibrant objects are precisely categorical equivalences.
\end{enumerate} 

If $(\C,\W_{\C})$ is a relative category, then the nerve $N\C$ is naturally a marked simplicial set with $\W_{\C}$ the marked edges. The underlying simplicial set of the fibrant replacement of $N(\C)$ in the Cartesian model structure will be denoted by $\C_\infty$. A map $\C\lrar \C'$ of relative categories is called an \textbf{equivalence} if $\C_\infty\lrar \C'_\infty$ is an equivalence of $\infty$-categories. We denote by $\RelCat$ the (large) relative category of relative categories. 

A \textbf{weak inverse} to a map of relative categories $F:\C\lrar \C'$ is a map of relative categories $G:\C'\lrar \C$ together with a zig-zag of natural transformations $\id\nRightarrow F\circ G$ and $\id\nRightarrow G\circ F$ which are (component-wise) weak equivalences. Note that if $F$ admits a weak inverse, then it is an equivalence of relative categories.    

\subsection{The category of $\infty$-groups}\label{ss:GPinf}
 
The relative category $\SS$ of Kan complexes is a standard model for the homotopy theory of $\infty$-groupoids. Similarly:

\begin{defn}
An \textbf{\boldmath$\infty$-group} is a pointed connected space and an $\infty$-group map is a map of pointed connected spaces. We denote by $\Gpinf:=\SS_*^{\geq 1}$ the relative category of $\infty$-groups.
\end{defn}

Our terminology is meant to suggest that an $\infty$-group is a higher categorical version of the notion of a group. However, it is also common in algebraic topology to view a loop space as a ``group up to (higher) homotopy". 
For the sake of completeness, let us prove that the homotopy theories of pointed connected spaces and of loop spaces, as defined here, are equivalent. 

\begin{defn}\
\begin{enumerate}
\item A \textbf{loop space} is a triple $(X, \B X,\mu)$ where $X$ is a space, $\B X$ is a pointed connected space and $\mu:X\overset{\simeq}{\lrar} \Om \B X$ is a homotopy equivalence.
\item
A \textbf{loop space map} from $(X,\B X,\mu)$ to $(Y,\B Y,\nu)$ is a pair consisting of a pointed map $\bar{\phi}:\B X\lrar \B Y$ and a map $\phi:X\lrar Y$ rendering the following square commutative
\begin{equation}\label{e:maps}
\xymatrix{X\ar[d]_{\mu}^{\simeq}\ar[r]^{\phi} & Y\ar[d]^{\nu}_{\simeq} \\\Om \B X \ar[r]^{\Om \bar{\phi}} &\Om \B Y .}
\end{equation}
\item
A loop space map is called an \textbf{equivalence} if $\phi$ and $\bar{\phi}$ are homotopy equivalences. We denote the \textbf{relative category} of loop spaces by $\loops$.
 
\end{enumerate}
\end{defn}

Let $\B:\loops\lrar \Gpinf$ be the relative functor given by $\B(X,\B X,\mu)=\B X$ and let $\Om:\Gpinf\lrar \loops$ be the relative functor given by $\Om (A):=(\Om A, A,\id_{\Om A})$. We have an adjunction of relative functors $$\xymatrix{\Om: \Gpinf\ar@<1ex>[r]& \loops:\B\ar@<1ex>[l]_{\upvdash}}$$ whose unit and counit transformations are homotopy equivalences in each component. Thus, $\Om$ is a weak inverse to $\B$ and we obtain the following:

\begin{cor}
The adjunction $\Om\dashv \B$ induces an equivalence of the underlying $\infty$-categories $\left(\Gpinf\right)_\infty\simeq\left(\loops\right)_\infty$. 
\end{cor}

\begin{rem} 
By means of the Kan loop group, any $\infty$-group can be replaced by a simplicial group. This may be viewed as a rectification procedure since the homotopy theory of simplicial groups is equivalent to that of pointed connected spaces. However, the existence of a rectification is \textbf{model dependent}: if we were to work in a relative category $(\C,\W)$, equivalent to $\SS$ there would be no reason that the induced homotopy theory of group objects in $\C$ will be equivalent to $\Gpinf$. For example if $(\Cat,\W_{\text{Thom}})$ is the relative category of categories and Thomason equivalences, then group objects in $\C=\Cat$ are precisely crossed modules, and their associated homotopy theory models only pointed connected $2$-types. 
\end{rem}

\section{The theory of $k$-invariants revisited}\label{s:k-invariants}

Our proof of the Sylow theorems for $\infty$-groups will require a strengthening of the classical theory of $k$-invariants to which we devote this section. Informally speaking, we wish to show that the $\infty$-category of pointed connected $(n+1)$-truncated spaces is equivalent to that of triples $(X,A,\kappa)$ where $X$ is a pointed connected $n$-truncated space, $A$ a $\pi_1X$-module and $\kappa$ is a map $$\xymatrix@C=1em@R=1em{X\ar@{-->}[rr]^-{\kappa}\ar[rd] && K(A,n+2)//\pi_1(X)\ar[dl]\\ & B\pi_1X.}$$ Although this result is known in folklore, we could not find a proof of it in the literature. We chose to give the details here as a service to the community.  

Let us briefly describe how the above-mentioned equivalence will be established. First, we will prove that the two $\infty$-categories at hand admit Cartesian fibrations (see \cite[2.4.2.1]{Lur09}) to the $\infty$-category of (pointed, connected) $n$-truncated spaces. We will then construct a map of Cartesian fibrations between them and prove that the induced map on fibers is a map of coCartesian fibrations over the category $\GpMod$ of abelian groups with an action of an arbitrary group. Lastly, we will use the classical results of \cite{DK84} to show that the latter map induces an equivalence on each fiber.

The verification that a map of $\infty$-categories is a (co)Cartesian fibration is usually an involved task. However, a result of Hinich (\cite[Proposition 2.1.4]{Hin}) gives simple conditions on a map of \textbf{relative categories} which assure that the underlying map on $\infty$-categories is a (co)Cartesian fibration. In order to use this method we will need to ``rectify" certain diagrams of $\infty$-categories to diagrams of relative categories. This is the reason for the various technical conditions that appear throughout this section.      

We first set up the relative category version of the Grothendieck construction. Recall that if $\F:\C\lrar \Cat$ is a (pseudo-)functor, the \textbf{Grothendieck construction} $\int_{\C}\F$ is the category where an object is a pair $(A,X)$ with $A\in\C$ and $X\in \F(A)$ and a morphism $(A,X)\lrar (B,Y)$ consists of a pair $(f,\phi)$ where $f:A\lrar B$ is a map in $\C$ and $\phi:f_!X:=\F(f)(X)\lrar Y$ is a map in $\F(B)$. The projection $\pi:\int_\C\F\lrar \C$ is then a \textbf{coCartesian fibration} classifying $\F$.
    
\begin{defn}
Let $(\C,\W_{\C})$ be a relative category and $\F:\C\lrar \RelCat$ a relative functor. The \textbf{Grothendieck construction} of $\F$ is the relative category $\left(\int_{\C} \F,\W_{\F}\right)$ where a map $(f,\phi):(A,X)\lrar (B,Y)$ is in $\W_\F$ if $f\in \W_{\C}$ (and thus $\phi$ is an equivalence of relative categories).    
\end{defn}       

Let $\C$ be a relative category and $\F:\C\lrar \RelCat$ a relative functor. Consider the underlying map of $\infty$-categories $\F_\infty:\C_\infty\lrar \Cat_\infty$ and the associated coCartesian fibration of $\infty$-categories $\int_{\C_\infty}\F_\infty\lrar \C_\infty$.
The following is a corollary of \cite[Proposition 2.1.4]{Hin}, adjusted for our purposes:

\begin{thm}\label{t:hin}
Let $\C$ and $\F:\C\lrar \RelCat$ be as above. Then there is a canonical categorical equivalence $\left(\int_\C\F\right)_\infty\overset{\simeq}{\lrar} \int_{\C_\infty}\F_\infty$ over $\C_\infty$.
\end{thm}
       
Let $\SS^{\geq 1,\leq n}_{*}$  be the relative category of pointed and connected $n$-truncated spaces (throughout, we assume $n\geq 1$). For a space $X\in \kan$ we denote by $X[n]:=\cosk_{n+1}X$ its $(n+1)$-coskeleton, which is a model for the $n^{\text{th}}$-Postnikov section of $X$. 
If $f:X \lrar Y$ is a map of spaces we denote by $f[n]:X[n] \lrar Y[n]$ the corresponding map on Postinkov sections.
A space $X$ is called \textbf{\boldmath$n$-coskeletal} if $X\cong\cosk_nX$. 

A space $X\in \SS$ is called \textbf{reduced} if $X_0=*$. To every space $X\in\SS_*^{\geq 1}$ we denote by $X_{\red}$ the reduced space obtained by the pullback
$$\xymatrix{X_{\red}\pbc\ar[r]\ar[d] & \{\ast\}\ar[d] \\ X\ar[r] & \cosk_0 X_0}$$
where $X_0$ is the constant simplicial set on the vertices of $X$. The functor $X\mapsto X_{\red}$ preserves fibrations which are surjective on $\pi_1$ and in particular takes Kan complexes to Kan complexes.
We let $\tau_{ n}\SS_{\red}^{\geq 1}\subseteq\SS_{*}^{\geq 1,\leq n}$ be the full subcategory spanned by reduced $(n+1)$-coskeletal spaces with its induced relative category structure. 

\begin{obs}\label{o:reduced}
The inclusion $\tau_{ n}\SS_{\red}^{\geq 1}\subseteq\SS^{\geq 1,\leq n}_{*}$ is an equivalence of relative categories.
\end{obs}   
  
Fix a space $X\in\tau_{ n}\SS_{\red}^{\geq 1}$. We denote by $G:=\pi_1(X)$ and by $BG$ the simplicial set given by the nerve of $G$. Note that there is a canonical fibration $X\twoheadrightarrow BG$. For a $G$-module $A$, we choose a (functorial) model for $K(A,n+2)//G$ as a Kan fibration over $BG$ (see e.g. \cite{DK84}).  Thus, both $X$
and $K(A,n+2)//G$ are objects in the category $\left({\SS}_{_{/BG}}\right)^{\fib}$ of fibrations over $BG$.

\begin{defn}

Let $X\in\tau_n\SS_{\red}^{\geq 1}$ and let $A$ be a $G$-module. The relative category $\K_n^X(A)$ has as objects the diagrams $$\xymatrix{X & Y\ar@{->>}[l]_\simeq\ar[r] & K(A,n+2)//G}$$ in $\left({\SS}_{_{/BG}}\right)^{\fib}$ with $Y\in \tau_{ n}\SS_{\red}^{\geq 1}$ . A morphism in $\K_n^X(A)$ is a diagram of the form

$$\xymatrix@R=1em@C=1em{& & Y\ar[dd]\ar[dr]\ar@{->>}[dll]_\simeq &\\ X & & & \;\;\;\;\;\;K(A,n+2)//G\\ & & Z\ar@{->>}[ull]^\simeq\ar[ur] &}$$ 
We take all maps in $\K_n^X(A)$ to be weak equivalences.
\end{defn}
    
Since $\left({\SS}_{_{/BG}}\right)^{\fib}$ is a category of fibrant objects, it admits a calculus of right fractions (\cite{Low}). It follows that the relative category $\K_n^X(A)$ models the corresponding mapping space:     
    
\begin{pro}[\cite{DK84}]
There is an equivalence of $\infty$-groupoids $$\K_n^X(A)_\infty\simeq \Map_{_{/BG}}(X,K(A,n+2)//G).$$
\end{pro}

A map of $G$-modules $A\lrar B$ gives rise to a map $$K(A,n+2)//G\lrar K(B,n+2)//G$$ (over $BG$) and thus the construction of $\K_n^X(A)$ extends to a functor $$\K_n^X:\G-Mod\lrar \RelCat$$ from the category of $G$-modules. 
\begin{defn}
The relative category of \textbf{$\textbf{n}^{\textbf{th}}$-level $\textbf{k}$-invariants of $\textbf{X}$} is the Grothendieck construction $\K_n(X):=\int_{A\in \G-Mod}\K_n^X(A)$. 
\end{defn}

For a space $X\in \tau_{ n}\SS_{\red}^{\geq 1}$ and a $G$-module $A$ as above, we also define: 

\begin{defn}\label{d: A-extension}
The category $\F^X(A)$ has as objects the triples $(X',\eta,\alpha)$ where $X'\in\tau_{ n+1}\SS_{\red}^{\geq 1}$, $\eta:X'\twoheadrightarrow X$ a fibration, $\alpha:\pi_{n+1}X'\cong A$ is an isomorphism of $G$-modules and such that the fibration $$X'[n]\overset{\simeq}{\twoheadrightarrow} X[n]\cong X$$ is a weak equivalence. A morphism is a commutative triangle of the form 

$$\xymatrix{X'\ar@{->>}[dr]\ar[rr] && X''\ar@{->>}[dl] \\ & X &}$$
suitably compatible with $\alpha$.
\end{defn}

\begin{rem}
Note that the map $X'[n]\lrar X[n]\cong X$ is automatically a fibration since $\cosk_{n+1}$ preserves Kan fibrations between Kan complexes if the codomain is $n$-truncated.
\end{rem}

Recall from \cite{DK84} that for any $X'\in \tau_{ n+1}\SS_{\red}^{\geq 1}$ with $G'=\pi_1(X')$ and $A'=\pi_{n+1}X'$ we have a canonical homotopy Cartesian square (over $BG$)
\begin{equation}\label{e:canonical}
\xymatrix{
X'\hpbc \ar[d]\ar[r]& X'[n]\ar[d]^-{k_{X'}}\\
BG'\ar[r] & K(A',n+2)//G'.
}
\end{equation}

The square (\ref{e:canonical}) the \textbf{\boldmath$n^\text{th}$-level canonical square} of $X'$.

Fixing the space $X\in \tau_{ n}\SS_{\red}^{\geq 1}$, Definition~\ref{d: A-extension} extends to a functor $$\F^X:\G-Mod\lrar \RelCat$$ as follows. If $f:A\lrar B$ is a map of $G$-modules, and $X'\twoheadrightarrow X$ is an object of $\F^X(A)$, we first extend the $n^{\text{th}}$-level canonical square of $X'$ to

$$\xymatrix{X'\hpbc\ar[r]\ar[d] & X'[n]\ar[d]\ar@{=}[rr] && X'[n]\ar[d]\\ BG'\ar[d]^{\simeq}\ar[r] &K(A',n+2)//G'\ar[rr]^{\alpha_*} && K(A,n+2)//G\ar[d]\\ \bar{BG'}\ar@{->>}[rrr] &&& K(B,n+2)//G,}$$ where we replaced the map $BG'\lrar K(B,n+2)//G$ by a fibration. 

We form the pullback 

$$\xymatrix@R=2em@C=2em{Y'_0\pbc\ar@{->>}[r]\ar[d] & X'[n]\ar[d]\\ \bar{BG'}\ar@{->>}[r] & K(B,n+2)//G,}$$ 

and define the space $f_!(X')=Y'$ to be $Y'_0[n+1]_{\red}$ with the evident fibration $$Y'\twoheadrightarrow X'[n]\twoheadrightarrow X$$ (note that $Y'_0[n+1]\lrar X[n]$ is a fibration which is sujective on $\pi_1$). 

Clearly,
$Y'\twoheadrightarrow X$ is an object in $\F^X(B)$ and we obtain a relative functor $$\F^X(A)\lrar \F^X(B).$$  

\begin{defn} 
The relative category of \textbf{\boldmath$(n+1)$-Postnikov extensions} of $X$ is the Grothendieck construction $\F(X):=\int_{A}\F^X(A)$.
\end{defn}

To define a relative functor $\K_n^X(A)\lrar \F^X(A)$, we first replacing the map $BG\lrar K(A,n+2)//G$ by a fibration. Then, for an object $$\xymatrix{X & Y\ar@{->>}[l]_\simeq\ar[r] & K(A,n+2)//G}$$ we form the pullback 
$$\xymatrix{Y_0'\pbc\ar[r]\ar@{->>}[d] & \bar{BG}\ar@{->>}[d]\\ Y\ar[r] & K(A,n+2)//G,}$$ with the evident map $Y_0'\twoheadrightarrow X$. The image of the object in $\F^X(A)$ is then the resulting map $Y':=Y_0'[n+1]_{\red}\twoheadrightarrow X$. 

The definition above combined with Theorem~\ref{t:hin} yields: 
\begin{pro}\label{p:coCart}
The induced map of $\infty$-categories $$\xymatrix@C=1em@R=1em{\K_n(X)_\infty\ar[rr]\ar[dr] && \F(X)_\infty\ar[dl] \\ & \G-Mod}$$ is a map of coCartesian fibrations. 
\end{pro}

Note that the $n^{\text{th}}$-level canonical pullback defines a relative functor $\F^X(A)\lrar \K_n^X(A)$ which is a weak inverse to $\K_n^X(A)\lrar \F^X(A)$. These functors are thus equivalences of relative categories and we obtain:

\begin{cor}\label{c:coCart}
The map $\K_n(X)_\infty \lrar \F(X)_\infty$ is an equivalence of $\infty$-categories.
\end{cor}   

%

We now wish to exhibit the map $\K_n(X)_\infty\lrar \F(X)_\infty$ as an induced map on fibers of \emph{Cartesian} fibrations over $\left(\SS^{\geq 1,\leq n}_*\right)_{\infty}$. 
Let us start by defining the functor $$\K_n:\left(\tau_n\SS^{\geq 1}_{\red}\right)^{\op}\lrar \RelCat.$$ If $X_0\lrar X_1$ is a map in $\tau_n\SS^{\geq 1}_{\red}$, denote $G_0=\pi_1(X_0),\; G_1=\pi_1(X_1)$ and define a relative functor $\K_n(X_1)\lrar \K_n(X_0)$ by sending an object $$\xymatrix{X_1 & Y_1\ar@{->>}[l]_\simeq\ar[r] & K(A,n+2)//G_1}$$ to the pullback $$\xymatrix{Y_0\pbc\ar@{->>}[r]^\simeq\ar[d] & X_0\ar[d]\\ Y_1\ar@{->>}[r]^\simeq & X_1.}$$ Note that $Y_0\twoheadrightarrow X_0$ is indeed an object of $\K_n(X_0)$ since we have a map of co-spans

$$\left(Y_1\lrar X_1 \llar X_0\right)\Rightarrow\left(K(A,n+2)//G_1\lrar BG_1 \llar BG_0\right)$$

Similarly, we define a functor $\F:\left(\tau_n\SS^{\geq 1}_{\red}\right)^{\op}\lrar \RelCat$ by sending a map $X_0\lrar X_1$ to the relative functor $\F(X_1)\lrar \F(X_0)$ which takes an object $X'_1\twoheadrightarrow X_1$ to the pullback $$\xymatrix{X'_0\pbc\ar[r]\ar@{->>}[d] & X'_1\ar@{->>}[d]\\ Y_1\ar[r] & X_1}.$$  
The dual of Theorem~\ref{t:hin} yields:
\begin{pro}\label{p:Cart}
The induced map of $\infty$-categories $$\xymatrix@R=1em@C=1em{\left(\int_{X}\K_n(X)\right)_\infty\ar[rr]\ar[dr] && \left(\int_X\F(X)\right)_\infty\ar[dl] \\ &\left(\tau_n\SS^{\geq 1}_{\red}\right)_{\infty}\simeq \left( \SS^{\geq 1,\leq n}_*\right)_\infty}$$ is a map of Cartesian fibrations.
\end{pro}
%
Corollary~\ref{c:coCart} yields:
\begin{cor}\label{c:Cart}
The map $\left(\int_{X}\K_n(X)\right)_\infty\lrar \left(\int_X\F(X)\right)_\infty$ is an equivalence of $\infty$ categories.
\end{cor}

We will denote by $$\K[n]:=\int_X\K_n(X)$$ the relative category defined above. Thus, an object in $\K[n]$ is an $(n+1)$-coskeletal reduced space $X$, together with a span $$\xymatrix{X & Y\ar@{->>}[l]_\simeq\ar[r] & K(A,n+2)//G}$$ where $G=\pi_1X$ and $A$ is a $G$-module. We refer to $\K[n]$ as the category of \textbf{\boldmath$n^{\text{th}}$-level $k$-invariants}.  
We are now ready to prove:

\begin{thm}\label{t:k-invariants}
There is an equivalence of $\infty$-categories $$\K[n]_\infty\simeq \left(\SS^{\geq 1,\leq n+1}_*\right)_{\infty}$$ 
\end{thm}

\begin{proof}
The projection $$\int_X \F(X)\lrar \tau_{n+1}\SS^{\geq 1}_{\red}$$ admits a weak inverse $$\tau_{n+1}\SS^{\geq 1}_{\red}\lrar \int_X \F(X)$$ which is given by sending an object $Y$ to the map $Y[n+1]_{\red}\twoheadrightarrow  Y[n]_{\red}$. 
The result now follows from Corollary~\ref{c:Cart} and Observation~\ref{o:reduced}.
%
\end{proof}

\section{Sylow theorems for $\infty$-groups}\label{s:Sylow}

Recall that an $\infty$-group is a pointed connected space that is, from our viewpoint, a homotopical analogue of a group. 
In this section we will study homotopical analogues of \textbf{finite} groups and parallel the Sylow theorems. 
\subsection{Basic notions}
\begin{define}\label{d:finite}\
\begin{enumerate}
\item An $\infty$-group $\GG$ will be called \textbf{finite} if all homotopy groups are finite.
We denote the relative category of such spaces by $\Fin$.
\item We say that $\GG \in \Fin$ is \textbf{\boldmath$n$-truncated} if $\pi_i(\GG)=0$ for $i>n$.

\end{enumerate}
\end{define}

\begin{defn}
Let $p$ be a prime. We say that $\GG \in \Fin$ is a \textbf{\boldmath$p$-\boldmath$\infty$-group} if all its homotopy groups are $p$-groups.
\end{defn}  

Henceforth, we \textbf{fix a prime $\mathbf{p}$}.
Let us first recall the (ordinary) group theoretic case. Recall that a $p$-Sylow subgroup of $G$ is a (possibly trivial) maximal $p$-subgroup of $G$. The Sylow theorems assert:
\begin{thm}[Sylow]\label{t:discrete syl}
Let $G$ be a finite group.
\begin{enumerate}
\item
For every $p$-group $H$ and a map $H\lrar G$, there is a $p$-Sylow subgroup $P_p\leq G$ and a factorisation $$\xymatrix@C=1em@R=1em{H\ar[rr]\ar[dr] && G\\ & P_p\ar[ur] &.}$$
\item 
The number of $p$-Sylow subgroups of $G$ is congruent to $1\;(mod\;p)$.
\item
All $p$-Sylow subgroups are conjugate to each other.
\end{enumerate}
\end{thm}
Theorem~\ref{t:syl} below is our analogue of Sylow theorems. 
To parallel Theorem~\ref{t:discrete syl} we first define a notion of a $p$-Sylow ``subgroup" for $\infty$-groups:
\begin{define}
A map $f:\PP\lrar \GG$ of finite $\infty$-groups, is called \textbf{\boldmath$p$-Sylow map} if for all $n\geq 1$, the map
$$\pi_n(f): \pi_n(\PP) \lrar \pi_n(\GG)$$
is an inclusion of a $p$-Sylow subgroup.
  
\end{define}

Let $\aut^h_*\GG$ be the (group-like) simplicial monoid of pointed self-equivalences of $\GG$. If $\GG$ is thought of as a pointed $\infty$-groupoid (with essentially one object), $B\aut^h_*\GG$ is a model for the \textbf{automorphism $\mathbf{\infty}$-group} of $\GG$ at the base object. We view he canonical map $\GG\lrar B\aut^h_*\GG$ as the \textbf{conjugation} map. If $\GG=BG$ is the classifying space of a discrete group then the induced map\\ 
$\pi_1 BG\lrar \pi_1 B\aut^h_*(BG)$ is precisely the conjugation map $\con :G\lrar \aut G$.  
\begin{defn}
A pair $f:\HH\lrar \GG,\;f':\HH'\lrar \GG$ of $\infty$-group maps are \textbf{conjugate} if there is an equivalence $\phi:\HH\overset{\simeq}{\lrar} \HH'$ and an equivalence $\psi:\GG\overset{\simeq}{\lrar} \GG$ such that
$$\xymatrix{\HH\ar[r]^{\phi}_{\simeq}\ar[d]_f & \HH'\ar[d]^{f'}\\ \GG\ar[r]_{\psi}^{\simeq} & \GG}$$ commutes and
 $[\psi]\in \pi_0\aut^h_*\GG$ is in the image of $$\pi_1\GG\lrar \pi_1\left(B\aut^h_*\GG\right)\cong\pi_0\aut^h_*\GG.$$
\end{defn}
\begin{rem}\label{r:conjugate}
Note that if $\GG=BG$ is the classifying space of a discrete group and $f:BH\lrar BG$ and $f':BH'\lrar BG$ are the induced maps of a pair $H,H'$ of subgroups of $G$, then $f$ is conjugate to $f'$ if and only if $H$ is conjugate to $H'$.
\end{rem}
\subsection{Factorisations of a $p$-$\infty$-group map by a $p$-Sylow map} 
Henceforth we fix an $\infty$-group $\GG$, a $p$-$\infty$-group $\HH$ and a map $f:\HH\lrar \GG$.
We denote by $\fact(f)$ the $\infty$-category of factorisations of $f$, given by the pullback of $\infty$-categories
$$\xymatrix{\fact(f)\pbc\ar[r]\ar[d] & \{f\}\ar[d]\\ \Fun\left(\Delta^2, \left(\Fin\right)_{\infty}\right)\ar[r]^{} & \Fun\left(\Delta^{\{0,2\}},\left(\Fin\right)_{\infty}\right)}$$  

When $\HH\simeq *$ is the trivial $\infty$-group and $f:*\lrar \GG$ is the obvious map, $\fact(f)$ is identified with the $\infty$-category $\left(\Fin_{/\GG}\right)_\infty$ of $\infty$-groups over $\GG$.
We denote by $\mathtt{P}_{f}$ the full sub-$\infty$-groupoid of $\fact(f)$ spanned by factorisations of the form
$$\HH \lrar \PP \lrar \GG$$
such that $\PP \lrar \GG$ is a $p$-Sylow map.
Note that the Postnikov functors $\P_n$ give us an infinite tower of $\infty$-categories,

$$\mathtt{P}_{f} \lrar \cdots \lrar \mathtt{P}_{f[n+1]} \lrar   \mathtt{P}_{f[n]} \lrar \cdots \lrar \mathtt{P}_{f[1]}  $$

and that $$\mathtt{P}_{f}\simeq \lim \mathtt{P}_{f[n]}  $$
since $(\SS_*^{\geq 1})_{\infty}\simeq \lim \P_n(\SS^{\geq 1}_*)_{\infty}$.

\begin{pro}\label{p:equivalence of groupoids}
For $n\geq 1$ the map
$$\mathtt{P}_{f[n+1]} \lrar  \mathtt{P}_{f[n]}$$ is an equivalence of $\infty$-groupoids 
 
\end{pro}

We will prove that the map in Proposition~\ref{p:equivalence of groupoids} is an equivalence by proving that all its fibers are contractible. Then the proposition is equivalent to the following statement:

\begin{pro}\label{p:factorization}
Let $\HH[n] \lrar \PP_0 \lrar \GG[n] \in \mathtt{P}_{f[n]}$. We denote by $\mathtt{S}$ the $\infty$-groupoid whose objects are the possible completions in $\left(\Fin\right)_{\infty}$ of the diagram
$$
\xymatrix{
\HH \ar[d] \ar[rr]^f & \empty & \GG \ar[d] \\
\HH[n] \ar[r] & \PP_0 \ar[r] &  \GG[n] }
$$
 to a diagram
$$
\xymatrix{
\HH \ar[d] \ar[r] & \PP\ar[d]\ar[r]^{} & \GG \ar[d] \\
\HH[n] \ar[r] & \PP_0 \ar[r]^{} &  \GG[n] }
$$

where $\PP \lrar \GG$ is a $p$-Sylow map. Then $\mathtt{S}$ is contractible.
\end{pro}

\begin{proof}

We first claim that $\mathtt{S}$ can be identified as the $\infty$-groupoid of lifts
$$\xymatrix{\Del^1\ar[r]^-{f}\ar[d] & \left(\SS^{\geq 1,\leq n+1}_*\right)_\infty\ar[d]^v\\ \Del^2\ar@{-->}[ur]^\ell\ar[r] & \left(\SS^{\geq 1,\leq n}_*\right)_\infty\times_{\Gp} \GpMod}$$ where $v$ is induced from the (relative) functor that associates to $\GG\in\SS^{1\geq, \leq n+1}_*$ the pair $(\GG[n], \pi_n \GG)$ over $\pi_1\GG$.
The lower-horizontal map is given by the pair of composable maps $$(\HH[n],A_\HH)\lrar (\PP_0,A_p)\lrar (\GG[n],A_\GG)$$ (here $A_\HH:=\pi_{n+1}\HH$, $A_\GG:=\pi_{n+1}\GG$ and $A_p$ is the (unique) $p$-Sylow subgroup of $A_\GG$). 

Let $G_p$ be the $p$-Sylow subgroup of $G=\pi_1(\GG)$ given by the image of $\pi_1(\PP_0)\lrar \pi_1(\GG)$. Using Theorem \ref{t:k-invariants} we can rewrite this as the $\infty$-groupoid of lifts
$$\xymatrix{\Del^1\ar[r]\ar[d] & \K[n]_\infty\ar[d]^{\bar{v}}\\ \Del^2\ar@{-->}[ur]^\ell\ar[r] & \left(\SS^{\geq 1,\leq n}_*\right)_{\infty}\times_{\Gp} \GpMod}$$
where $\bar{v}$ is induced from the relative functor that associates to a triple $$\left(X,A,\eta:X\overset{\simeq}{\twoheadleftarrow} Y\lrar K(A,n+2)//\pi_1(X) \right)$$ the pair $(X,A)$.

By construction, the fiber of $\bar{v}$ over $(X,A)$ is the mapping space $$\Map_{/B\pi_1X}\left(X,K(A,n+2)//\pi_1X\right)$$ and it follows that we can further identify $\mathtt{S}$ with the space of maps $$\PP_0\lrar K(A_p,n+2)//G_p$$ which fit into the following diagram in $\left(\Fin\right)_\infty$:

$$\xymatrix@R=1.5em@C=1.5em{\HH[n]\ar[dd]\ar[rr]\ar[dr]  & & K(A_\HH,n+2)//\pi_1(\HH)\ar[dd]\ar[dl]\\
&B\pi_1(\HH)\ar[dd]|\hole&\\
 \PP_0\ar@{-->}[rr]\ar[dd]\ar[dr]& & K(A_p,n+2)//G_p\ar[dd]\ar[dl]\\ 
 &BG_p\ar[dd]|\hole&\\
 \GG[n]\ar[dr]\ar[rr]& & K(A_\GG,n+2)//G \ar[dl]\\
 &BG&}$$ (all maps are over $BG$). 
 Note that we have a homotopy Cartesian square:
 $$\xymatrix{
 K(A_p,n+2)//G_p\hpbc\ar[r]\ar[d] & BG_p\ar[d]\\
 K(A_\GG,n+2)//G \ar[r] &  BG 
} $$
 Thus $\mathtt{S}$ is just the space of dashed lifts in the $\infty$-category of spaces over $BG_p$:

$$\xymatrix{\HH[n]\ar[r]\ar[d] & K(A_p,n+2)//G_p\ar[d]\\ \PP_0\ar@{-->}[ur]^\ell\ar[r] & K(A_\GG,n+2)//G_p}$$


As a $G_p$-module, $A_\GG$ factors naturally as a sum 
$$A_\GG = A_p \oplus A^{(p)}$$
where $A^{(p)}$ is the prime-to-$p$ part. Now, in the $\infty$-category of spaces over $BG_p$ we have 
$$K(A_\GG,n+2)//G_p \simeq K(A_p,n+2)//G_p \times K(A^{(p)},n+2)//G_p $$ so that $\mathtt{S}$ can be identified as the space of dashed lifts in $\left(\SS_{/BG_p}\right)_\infty$:

$$\xymatrix{\HH[n]\ar[r]\ar[d] & BG_p \ar[d]\\ \PP_0\ar@{-->}[ur]^\ell\ar[r] & K(A^{(p)},n+2)//G_p}$$

This space of lifts is in turn just the homotopy fiber of the map

$$
\xy<0.6cm,0cm>:
\POS (0,0) *+{\Map_{/BG_p}(\PP_0,K(A^{(p)},n+2)//G_p)\times_{\Map_{/BG_p}(\HH[n],K(A^{(p)},n+2)//G_p)}\Map_{/BG_p}(\HH[n],BG_p)}="a";
\POS (5,1.5) *+{(\star)};
\POS (0,3) *+{\Map_{/BG_p}(\PP_0,BG_p)}="b";
\POS "b" , \ar@{->} "a"
\endxy 
$$

We will show that all spaces above are contractible.
Firstly, the spaces \\$\Map_{/BG_p}(\HH[n],BG_p) $ and $\Map_{/BG_p}(\PP_0,BG_p) $ are contractible by definition.
Secondly, we have isomorphisms
$$\pi_i(\Map_{/BG_p}(\HH[n],K(A^{(p)},n+2)//G_p )) = \tilde{H}^{n+2-i}(\HH[n],A^{(p)})$$ 
$$\pi_i(\Map_{/BG_p}(\PP_0,K(A^{(p)},n+2)//G_p )) = \tilde{H}^{n+2-i}(\PP_0,A^{(p)}).$$   
Since $\PP_0$ and $\HH[n]$ are $p$-spaces, and $A^{(p)}$ is prime-to-$p$,  
these cohomology groups vanish by a Serre class argument. 
It follows that the map $(\star)$ is a map between contractible spaces and in particular has a contractible homotopy fiber $\mathtt{S}_\infty$.
\end{proof}

\begin{notn}
For a finite $\infty$-group $\GG$ and $f:\ast\lrar \GG$, we let $\mathtt{P}_p:=\mathtt{P}_f$ be the (homotopy discrete) $\infty$-groupoid whose objects are $p$-Sylow maps $\PP\lrar \GG$. 
\end{notn}

We are now at state to prove our analogue of the Sylow theorems.
\begin{thm}\label{t:syl}
Let $\GG$ be a finite $\infty$-group and let $p$ be a prime.
\begin{enumerate}
\item The $\infty$-groupoid $\mathtt{P}_p$ of $p$-Sylow maps $\PP\lrar \GG$ is equivalent to the \textbf{set} of $p$-Sylow subgroups of $\pi_1(\GG)$. In particular, $\mathtt{P}_p\neq \emptyset$ and $|\pi_0\mathtt{P}_p|\equiv 1\;(mod\; p)$.
\item
For every finite $p$-$\infty$-group $\HH$, any map $f:\HH\lrar \GG$ factors as $$\xymatrix{\HH\ar[rr]\ar[dr] && \GG\\ &\PP\ar[ur]& }$$ where $\PP\lrar \GG$ is a $p$-Sylow map.
\item
Every two $p$-Sylow maps are conjugate. 
\end{enumerate}
 
\end{thm}

\begin{rem}
For a finite group $G$, the classical Sylow theorems appearing in~\ref{t:discrete syl} can be obtained from Theorem~\ref{t:syl} above with $\GG:=BG$ (for part (3) see Remark~\ref{r:conjugate}). Note also that a $p$-Sylow subgroup $P$ of $G$ will generally have an interesting automorphism group.   
However, when $P$ is considered as an object in the category $\Gp^{\fin}_{/G}$ of finite groups over $G$ (via the inclusion $P\hrar G$), it admits no non-trivial automorphism. This is the reason we chose to work in the $\infty$-category $\left(\Fin_{/\GG}\right)_\infty$. 
\end{rem}

\begin{proof}

From Proposition~\ref{p:factorization} we conclude that
$$\mathtt{P}_f \simeq \mathtt{P}_{f[1]}.$$

Now $\mathtt{P}_{f[1]}$ is the $\infty$-groupoid of factorisations
$$\HH[1] \lrar \PP \lrar \GG[1].$$
Since the objects of such factorisations are all $1$-types, $\mathtt{P}_{f[1]}$ is in fact discrete and is equivalent to the \textbf{set} of $p$-Sylow subgroups of $\pi_1(\GG)$ that contain the image of $$\pi_1(f):\pi_1(\HH) \lrar \pi_1(\GG).$$

For $\HH= *$ we get that the $\infty$-groupoid of $p$-Sylow maps $\PP \lrar \GG$
is discrete and is equivalent to the \textbf{set} of $p$-Sylow subgroups
of $\pi_1(\GG)$. This proves $(1)$ and $(2)$.
To prove $(3)$, fix a $p$-Sylow subgroup $P_p\hrar G$. Since all $p$-Sylow subgroups in $\pi_1(\GG)$ are conjugate, every $p$-Sylow map $\PP_p\lrar \GG$ is conjugate to a $p$-Sylow map $\bar{\PP_p}\lrar \GG$ such that the image of $\pi_1\bar{\PP_p}\lrar \pi_1\GG$ is $P_p$. But now, all such $p$-Sylow maps are equivalent as objects of $\mathtt{P}_f$ since $$\mathtt{P}_f \simeq \mathtt{P}_{f[1]}.$$   
\end{proof}

We would like to point out that the analogy between Theorem \ref{t:syl} and the classical Sylow theorems is not complete. Recall that an ordinary finite group $G$ admits a unique $p$-Sylow subgroup if and only if the $p$-Sylow subgroup is normal. In the context of $\infty$-groups, there is a notion of normality which was introduced in \cite{FS} and developed in \cite{Pra} for loop spaces. Let us spell-out the corresponding notion for $\infty$-groups.    
\begin{defn}
A map of $\infty$-groups $f:\NN\lrar \GG$ is called \textbf{normal} if there is an $\infty$-group $\QQ$ and a map $\GG\lrar \QQ$ such that $\NN\lrar \GG\lrar \QQ$ is a homotopy fiber sequence. The data of the map $\pi:\GG\lrar \QQ$ and the (pointed) null-homotopy $\pi\circ f\simeq *$ is called a \textbf{normal structure}. We will also denote $\QQ:=\GG//\NN$ when the normal structure is clear from the context.   
\end{defn}

\begin{example} Although for an ordinary group, having  a unique $p$-Sylow subgroup is equivalent to the $p$-Sylow subgroup being normal, the analogous ``higher statement" is \bf{not true}. Namely there exists a finite $\infty$-group $\GG$ with a contractible $\infty$-groupoid of $p$-Sylow maps for which the essentially unique $p$-Sylow map $\PP_p \to \GG$ is not homotopy normal. For example, let $A$ be the unique $\mathbb{Z}/2$-module with $3$-elements and a non-trivial action.
Let $$X  = K(A,2)//(\mathbb{Z}/2).$$
we have $\pi_1(X) = \mathbb{Z}/2$  which is abelian. Thus there exist a unique $2$-sylow map
$$B(\mathbb{Z}/2) \to  K(A,2)//(\mathbb{Z}/2).$$
However this map is not homotopy normal. Indeed if it was we could have completed the diagram by adding
$$B(\mathbb{Z}/2) \to  K(A,2)//(\mathbb{Z}/2) \to K(\mathbb{Z}/3,2).$$
 This cannot happen by the functoriality of homotopy groups as $\pi_1$-modules.
\end{example}

\subsection*{A remark about $p$-completion.} When an $\infty$-group $\GG$ is nilpotent, the composite $\PP_p\lrar \GG\lrar \GG^{\wedge}_p$ of a $p$-Sylow map and the $p$-completion is an equivalence by Corollary~\ref{c:p-completion}. Thus, in this case, constructing a homotopy section to $\GG\lrar \GG^{\wedge}_p$ is the same as constructing a $p$-Sylow map. As was pointed out to us by Jesper Grodal, using this observation one can take an alternative approach for proving Theorem~\ref{t:syl}. Although this can be considered as a more direct construction of a $p$-Sylow map, we find the current proof of Theorem~\ref{t:syl} more suitable to the higher-categorical point of view taken in this paper.

\section{Applications}\label{s:applications}
We call a space \textbf{finite} if all its homotopy groups are finite.
In this section we will illustrate how Theorem~\ref{t:syl} can be applied. Our first application will be to establish an analogue of the Burnside's fixed point lemma. Our second application will be a characterisation of finite nilpotent spaces (viewed as $\infty$-groups) in terms of Sylow maps into them.

\subsection{Burnside's fixed point lemma for $p$-$\infty$-groups}\label{ss:Burnside}
Let us first recall a well-known corollary of Burnside's lemma:
\begin{pro}\label{p:burnside}
Let $G$ be a $p$-group and $X$ a finite $G$-set. If the order of $X$ is prime-to-$p$ then the fixed point set $X^G$ is non-empty. 
\end{pro}

The parallel statement is now:
\begin{pro}
Let $\GG$ be a finite $p$-$\infty$-group. Let $X$ be a finite space with an action of the Kan loop group $\K(\GG)$ of $\GG$. If all homotopy groups of $X$ are of prime-to-$p$ order, then the space of homotopy fixed points $X^{h\K(\GG)}$ is non-empty. 
\end{pro}

\begin{proof}
If $X$ is non-connected, Proposition~\ref{p:burnside} implies that there is a connected component of $X$ which is fixed by $\pi_0(\K(\GG))$ and thus by $\K(\GG)$. Hence, we may assume without loss of generality that $X$ is connected.   
We choose a point in $X$ so that we have a homotopy fiber sequence of pointed spaces:
$X\lrar X//\K(\GG)\overset{\gamma}{\lrar} \GG$. Recall that a homotopy fixed point can be equivalently described as a homotopy section $s:\GG\lrar X//\K(\GG)$ of $\gamma$. The long exact sequence takes the form 
$$...\lrar\pi_n(X)\lrar \pi_n(X//\K(\GG))\overset{c_n=\pi_n(\gamma)}{\lrar} \pi_n(\GG) \lrar \pi_{n-1}(X)\lrar...$$ which implies that $ker(c_n)$ and $coker(c_n)$ are of prime-to-$p$ order. Hence, $c_n$ maps every $p$-Sylow subgroup of $\pi_n(X//\K(\GG))$ isomorphically onto a $p$-Sylow subgroup of $\pi_n(\GG)$. If we now choose a $p$-Sylow map $\PP_p\lrar X//\K(\GG)$ we see that the composite map $\PP_p\lrar X//\K(\GG)\lrar \GG$ is a weak equivalence by the assumption on $\GG$. We thus have a homotopy section in the form $\GG\simeq \PP_p\lrar X//\K(\GG)$.     
\end{proof}

\subsection{finite nilpotent $\infty$-groups}\label{ss:nilpotent}

We finish this section with a result on nilpotency. Recall \cite[II.4]{BK} that a pointed connected space $X$ is \textbf{nilpotent} if $\pi_1X$ is nilpotent and acts nilpotently on $\pi_n(X)$ for all $n\geq 2$. 

\begin{defn}
An $\infty$-group $\GG$ is called \textbf{nilpotent} if its underlying space is nilpotent.
\end{defn}

Our goal in this subsection is to prove a parallel to the following classical theorem (see e.g. \cite[Theorem $5.39$]{Rot})
\begin{thm}\label{t:nilpotent} Let $G$ be a finite group. The following conditions are equivalent:
\begin{enumerate}
\item
$G$ is nilpotent.
\item
$G$ is isomorphic to the product of its Sylow subgroups.
\item
All Sylow subgroups of $G$ are normal. 
\end{enumerate}
\end{thm}

\begin{cor}\label{c:p-completion}
Let $\GG$ be a finite nilpotent $\infty$-group and consider its $p$-completion $\GG\lrar \GG^{\wedge}_p$. If $\PP_p\lrar \GG$ is a (essentially unique) $p$-Sylow map, then the composite $\PP_p\lrar \GG\lrar \GG^{\wedge}_p$ is an equivalence. 
\end{cor}

Let $\GG$ be a finite nilpotent $\infty$-group. Since $G=\pi_1(\GG)$ is a finite nilpotent group there is a unique $p$-Sylow subgroup $P_p\leq G$ for every prime $p$ with $p|G$. It follows that for every prime $p$, the relative category $\mathtt{P}_p$ of $p$-Sylow maps $\PP\lrar \GG$ and equivalences is weakly contractible. In other words, for each prime $p$, there is an essentially unique $p$-Sylow map $\PP_p\lrar \GG$.  

\begin{defn}
Let $\GG$ be a finite $\infty$-group. A collection of Sylow maps \\$\{\PP\lrar \GG\}$ is called \textbf{ample} if no two maps are equivalent (over $\GG$) and the collection $\{\pi_1(\PP)\hrar \pi_1(\GG)\}$ equals the set of $p$-Sylow subgroups of $G$, for all primes $p|G$.
\end{defn}

\begin{thm}\label{t:nil}
Let $\GG$ be a finite $\infty$-group with $G=\pi_1(\GG)$ and let $\{\PP_{p}\lrar \GG\}_{p|G}$ be an ample collection of Sylow maps into $\GG$. The following conditions are equivalent:
\begin{enumerate}
\item
$\GG$ is nilpotent. 
\item
$\GG\simeq \prod_{p|G}\PP_{p}.$
\item
For every prime $p|G$, the map $\PP_{p}\lrar \GG$ is normal.
\end{enumerate}
\end{thm}

Before going to the proof of Theorem~\ref{t:nil}, let us recall an algebraic fact:
\begin{lem}\label{l:nil}
Let $H$ be a $p$-group and $M$ a finite $H$-module  of a $p$-power order. Then $M$ is nilpotent  as a $H$-module.
\end{lem}

\begin{proof}[Proof of Theorem~\ref{t:nil}]
\begin{enumerate}
\item 
$(2)\Rightarrow (3):$ Clear.
\item
$(3)\Rightarrow (1):$ In that case, $\pi_1(\PP_{p})$ is normal in $\pi_1(\GG)$ by the LES for $\PP_{p}\lrar \GG\lrar \GG//\PP_{p}$ so that all the Sylow subgroups of $G=\pi_1(\GG)$ are normal and $G=\pi_1(\GG)\cong\displaystyle\mathop{\prod}_{p|G} \pi_1(\PP_{p})$ is thus nilpotent. Note that $\pi_n(\GG)\cong \displaystyle\mathop{\prod}_{p|G} \pi_n(\PP_{p}) $

The LES for $\PP_{p}\lrar \GG\lrar \GG//\PP_{p}$ also implies  that $\pi_1(\PP_p)$ must act trivially on the prime-to-$p$ part of each $\pi_n(\GG)$. Using Lemma~\ref{l:nil} with $H=\pi_1(\PP_{p})$ we deduce that $\GG$ is nilpotent.
\item 
$(1)\Rightarrow (2):$ Suppose first that $\GG$ is $n$-truncated. We argue by induction on $n$. If $n=1$, the decomposition $G=\pi_1(\GG)\cong \displaystyle\mathop{\prod}_{p|G}\pi_1(\PP_{p})$ implies a decomposition $\GG\simeq \prod_{p|G} \PP_{p}.$ Suppose $\GG$ is of homotopical dimension $n+1$. We denote $A:=\pi_{n+1}(\GG)$, $A_{p}:=\pi_{n+1}(\PP_{p})$ and $G_{p}:=\pi_1(\PP_{p})$. The canonical homotopy Cartesian square of~\ref{e:canonical} takes the form
\begin{equation}\label{e:nil}
\xymatrix{
\GG \ar[d]\pbc\ar@{}[d]|(0.30){\;\;\;\;\;\;\;\;\;\;\;\sim}\ar[r]& B\pi_1(\GG)\ar[d]\\
\GG[n] \ar[r]^-{k} & K(A,n+2)//\pi_1(\GG).
}
\end{equation}
Since $\GG$ is nilpotent, the terms in the right-upper and right-lower part of the diagram decompose into a product and the induction assumption implies that the left-lower term decomposes as well. The homotopy Cartesian square of~\ref{e:nil} thus takes the form
$$\xymatrix{
\GG \ar[d]\pbc\ar@{}[d]|(0.30){\;\;\;\;\;\;\;\;\;\;\;\sim}\ar[r]& \prod_{p|G} BG_{p}\ar[d]\\
\prod_{p|G} P_n(\PP_{p}) \ar[r]^-{k} & \prod_{p|G} K(A_{p},n+2)//G_{p}.}$$ 
The right-vertical map clearly decomposes to the product of the corresponding maps. The lower-horizontal map decomposes as well, by a Serre class argument.
It follows that $\GG\simeq \prod_{p|G} \PP_{p}.$ 

If $\GG$ is an arbitrary finite $\infty$-group, then it is a limit of its $n$-truncations in the $\infty$-category $\left(\Fin\right)_\infty$ and the statement follows.
\end{enumerate}
\end{proof}

\end{document}